\documentclass[12pt]{article}

\usepackage{amsfonts}
\usepackage{amsthm}

\newtheorem{theorem}{Theorem}[section]
\newtheorem{lemma}{Lemma}[section]
\newtheorem{corollary}{Corollary}[section]

\makeatletter
\@addtoreset{equation}{section}
\makeatother

\begin{document}

 \sloppy

\begin{center}
{\bf \Large\bf On one condition of absolutely continuous spectrum for self-adjoint operators and its applications.}
\\ \bigskip
{\large Ianovich E.A.}
\\ \bigskip
   {\it Saint Petersburg, Russia}
\\ \bigskip
   {\it\small E-mail: eduard@yanovich.spb.ru}
\end{center}

\begin{abstract}
In this work the method of analyzing of the absolutely continuous spectrum for self-adjoint operators is considered. For the analysis it is used an approximation of self-adjoint operator $A$ by a sequence of operators $A_n$ with absolutely continuous spectrum on a given interval $[a,b\,]$ which converges to $A$ in a strong sense on a dense set. The notion of equi-absolute continuity is also used. It was found a sufficient condition of absolute continuity of the operator $A$ spectrum on the finite interval $[a,b\,]$ and the condition for that the corresponding spectral density belongs to the class $L_p[a,b\,]$ ($p\ge 1$).
The application of this method to Jacobi matrices is considered. As a one of the results we obtain the following assertion: 
Under some mild assumptions (see details in Theorem (2.4)), suppose that there exist a constant $C>0$ and a positive function $g(x)\in L_p[a,b\,]$ ($p\ge1$) such that for all $n$ sufficiently large and almost all $x\in[a,b\,]$ the estimate 
$\frac{\displaystyle 1}{\displaystyle g(x)}\le b_n(P_{n+1}^2(x)+P_{n}^2(x))\le C$ holds,
where $P_n(x)$ are 1st type polynomials associated with Jacobi matrix (in the sense of Akhiezer) and $b_n$ is a second diagonal sequence of Jacobi matrix. Then the spectrum of Jacobi matrix operator is purely absolutely continuous on $[a,b\,]$ and for the corresponding spectral density $f(x)$ we have $f(x)\in L_p[a,b\,]$.
\end{abstract}

\bigskip

\section{Statement of the problem and main results.}

Let us consider the following problem. Let $\{A_n\}_1^\infty$ be a sequence of linear self-adjoint operators in a separable Hilbert space $H$ with absolutely continuous spectrum on some interval $[a,b\,]$ and $\Phi$ be some dense set in $H$. Suppose that on $\Phi$ there exists in a strong sense the limit $\lim\limits_{n\to\infty}A_n$ and the closure of this limit operator is self-adjoint operator A. That is for any $e\in\Phi$
$$
\lim\limits_{n\to\infty}\|(A_n-A)e\|=0
$$

Suppose also that $[a,b\,]\subseteq\sigma(A)$. There arises the question: what is the condition which guarantees that the spectrum of a limit operator $A$ is also absolutely continuous on $[a,b\,]$ ?

Let $E_\lambda^{(n)}$ and $E_\lambda$ be the resolutions of the identity of the operators $A_n$ and $A$ respectively and let
$$
\sigma_n(\lambda;\,e)=(E_\lambda^{(n)} e, e),\quad e\in H
$$
$$
\sigma(\lambda;\,e)=(E_\lambda e, e),\quad e\in H
$$
By assumption all functions $\sigma_n(\lambda;\,e)$ are absolutely continuous on $[a,b\,]$ so that
$$
\sigma_n(\lambda;\, e)=\sigma_n(a;\, e)+\int\limits_{a}^\lambda f_n(t; e)\,dt\,,\quad e\in H\,,\quad \lambda\in[a,b\,],
$$
where $f_n(t;\,e)$ is a sequence of non-negative summable on $[a,b\,]$ functions.

It may to suppose that the convergence of the sequence $A_n$ to the operator $A$ in a strong sense on the set $\Phi$ should leads to the convergence of the sequence $\sigma_n(\lambda; e)$ to the function $\sigma(\lambda; e)$. It turns out it is true.
\begin{lemma}
\label{lemma2}
If $\lambda$ is not an eigenvalue of the operator $A$, then for all $e\in H$
\begin{equation}
\label{sigma_n_to_sigma}
\lim\limits_{n\to\infty}\sigma_n(\lambda; e)=\sigma(\lambda; e)
\end{equation}
\end{lemma}

\begin{proof} It is known~\cite{1} that the strong convergence of the sequence $A_n$ to $A$ on a dense subset in $H$ leads to the strong convergence of the sequence $E_\lambda^{(n)}$ to $E_\lambda$ for those $\lambda$ which are not the eigenvalues of $A$. That is if $\lambda$ is not an eigenvalue of the operator $A$ then for all $e\in H$
$$
\lim\limits_{n\to\infty}\|(E_\lambda^{(n)}-E_\lambda)e\|=0
$$
Further, we have
$$
|\sigma_n(\lambda; e)-\sigma(\lambda; e)|=|((E_\lambda^{(n)}-E_\lambda)e, e)|\le\|(E_\lambda^{(n)}-E_\lambda)e\|\,,
$$
whence the statement of the Lemma readily follows.
\end{proof}

Note that as we suppose that the space $H$ is separable, the set of eigenvalues of the operator $A$ is no more than denumerable and hence the equality~(\ref{sigma_n_to_sigma}) is fulfilled for almost all $\lambda$.

For further consideration we need the notion of equi-absolute continuity~\cite{2,3}. The sequence $g_n(x)$ of absolutely continuous on the interval $[a,b\,]$ functions
$$
g_n(x)=g_n(a)+\int\limits_a^x f_n(t)\,dt
$$
is called equi-absolutely continuous on this interval if for any $\epsilon>0$ there exists $\delta>0$ (which is independent of $n$) such that for any measurable subset $e\subset[a,b\,]$ the inequality
$$
\left|\int\limits_e f_n(t)\,dt\right|<\epsilon
$$
is valid, provided $m(e)<\delta$ ( $m(\cdot)$ is Lebesgue measure ). The sequence of summable functions $f_n(t)$ is said to be have equi-absolutely continuous integrals on $[a,b\,]$.

Let us give now the answer on the question stated above.
\begin{theorem}
\label{criterion_of_absolutely_continuity}
In order that the spectrum of the operator $A$ on the interval $[a,b\,]$ $(\,[a,b\,]\subseteq\sigma(A)\,)$ be absolutely continuous it is sufficient that for any vector $e$ from generating subspace of the operator $A$ the sequence of functions $\{\sigma_n(\lambda;\, e)\}$ be equi-absolutely continuous on $[a,b\,]$, or that is the same, the sequence of functions
$\{f_n(t; e)\}$ has equi-absolutely continuous integrals on $[a,b\,]$. Under these conditions the sequence $\{\sigma_n(\lambda;\, e)\}$ uniformly converges to $\sigma(\lambda;\, e)$ on $[a,b\,]$.
\end{theorem}

\begin{proof} Fix some vector $e$ from generating subspace of $A$. For the brevity we will omit this vector in the arguments of all functions.

Consider the sequence of functions $\sigma_n(\lambda)$ on the interval $[a,b\,]$.
$$
\sigma_n(\lambda)=\sigma_n(a)+\int\limits_{a}^\lambda f_n(t)\,dt\,,\quad \lambda\in[a,b\,]
$$

Here $f_n(t)\ge0$ on $[a,b\,]$ but the case $f_n(t)\equiv0$ is also possible if for example the corresponding vector $e$ belongs to the discrete subspace of the operator $A_n$. In this case the function $\sigma_n(\lambda)$ is constant on $[a,b\,]$.

From Lemma~(\ref{lemma2}) we conclude that for almost all $\lambda\in[a,b\,]$ $\lim\limits_{n\to\infty}\sigma_n(\lambda)=\sigma(\lambda)$.
By assumption the sequence $\sigma_n(\lambda)$ is equi-absolutely continuous on $[a,b\,]$. Besides, by virtue of its definition, it is uniformly bounded ($0\le\sigma_n(\lambda)\le\|e\|^2\,,\, \lambda\in R$). Hence~\cite{2}, from the sequence $\sigma_n(\lambda)$ one can extract a partial subsequence $\sigma_{n_k}(\lambda)$ which uniformly converges on $[a,b\,]$, moreover the limit function $\sigma_1(\lambda)$ is absolutely continues on $[a,b\,]$. So, we obtain
$$
(*)\quad\lim\limits_{n\to\infty}\sigma_n(\lambda)=\sigma(\lambda)\quad\mbox{for almost all $\lambda$ from $[a,b\,]$}
$$
$$
\lim\limits_{k\to\infty}\sigma_{n_k}(\lambda)=\sigma_1(\lambda)\quad\mbox{for all $\lambda$ from $[a,b\,]$}
$$
Comparing this two limit equalities we can note that $\sigma_1(\lambda)=\sigma(\lambda)$ for almost all $\lambda\in[a,b\,]$. Let us show that the first equality (*) is fulfilled everywhere in $[a,b\,]$. Let $\lambda_0$ be a point in which the equality (*) perhaps is not valid. Consider the sequence $\sigma_n(\lambda_0)$. We have
$$
|\sigma_n(\lambda_0)-\sigma_m(\lambda_0)|\le|\sigma_n(\lambda_0)-\sigma_n(\lambda')|+|\sigma_n(\lambda')-\sigma_m(\lambda')|+
|\sigma_m(\lambda')-\sigma_m(\lambda_0)|
$$
Here $\lambda'$ is an arbitrary point of the interval $[a,b\,]$. Since the set of the possible points $\lambda_0$ has measure zero, in any arbitrarily small neighbourhood of the point $\lambda_0$ there exists an infinite number of points in which the sequence $\sigma_n(\lambda)$ converges. Hence one can choose the point $\lambda'$ such that on the one hand the sequence $\sigma_n(\lambda')$ converges and on the other hand the inequality
$$
|\sigma_n(\lambda_0)-\sigma_n(\lambda')|<\epsilon/3
$$
holds for any $\epsilon>0$ and any $n$.

The validity of this inequality for any $n$ is possible by virtue of equi-absolute continuity of the sequence $\sigma_n(\lambda)$.

Fixing such $\lambda'$ and $\epsilon$, one can choose a number $N$ such that the inequality
$$
|\sigma_n(\lambda')-\sigma_m(\lambda')|<\epsilon/3
$$
is valid, provided $n,m>N$. Then for the same $n$ and $m$
$$
|\sigma_n(\lambda_0)-\sigma_m(\lambda_0)|<\epsilon
$$
Since $\epsilon$ is arbitrarily small the sequence $\sigma_n(\lambda_0)$ converges by Cauchy's criterion.

Thus the sequence $\sigma_n(\lambda)$ converges everywhere in $[a,b\,]$ and therefore $\sigma(\lambda)=\sigma_1(\lambda)$ also everywhere in $[a,b\,]$. Hence the function $\sigma(\lambda)$ is absolutely continuous on this interval. Note that the sequence $\sigma_n(\lambda)$ converges to $\sigma(\lambda)$ uniformly on $[a,b\,]$. It can be proved by contradiction. Actually suppose it is not true. Then there exist a positive $\epsilon$, subsequence $n_k$ and corresponding to it points $x_k\in[a,b\,]$ so that
\begin{equation}
\label{contradiction}
|\sigma_{n_k}(\lambda_k)-\sigma(\lambda_k)|\ge\epsilon(>0)\,,\quad k=0,1,2,\ldots
\end{equation}
Sequence $\sigma_{n_k}(\lambda)$ contains uniformly convergent subsequence $\sigma_{n_{k_m}}(\lambda)$ in $[a\,b\,]$ and the limit function is $\sigma(\lambda)$. Hence
$$
|\sigma_{n_{k_m}}(\lambda)-\sigma(\lambda)|<\epsilon\,,
$$
provided $\lambda\in[a,b\,]$ and $n_{k_m}>N$. In particular 
$$
|\sigma_{n_{k_m}}(\lambda_{k_m})-\sigma(\lambda_{k_m})|<\epsilon
$$
for $m$ sufficiently large. But this inequality contradicts to~(\ref{contradiction}), and the uniform convergence is proved.

Thus we proved that for any vector $e$ from generating subspace of $A$ the functions $\sigma(\lambda)\equiv\sigma(\lambda; e)$ are absolutely continuous on $[a,b\,]$. Since $[a,b\,]\subseteq\sigma(A)$ it follows already that the spectrum of the operator $A$ is absolutely continuous on this interval. The theorem is proved.
\end{proof}

\begin{corollary}
\label{criterion_of_absolutely_continuity_corollary}
If for any vector $e$ from generating subspace of $A$ the sequence of functions $\{\sigma_n(\lambda;\, e)\}$ is equi-absolutely continuous on any finite interval of a real axis then the spectrum of the operator $A$ is purely absolutely continuous.
\end{corollary}

In~\cite{4} it was used in fact a simple sufficient condition of equi-absolute continuity when $f_n(t)\le g(t)$ where
$g(t)\in L_\infty[a,b\,]$. One can easily generalize it on the case of any $L_p,\, p\ge1$.

\begin{lemma}
\label{equi-absolute_condition}
Suppose there exists a positive function $g(t)\in L_p[a,b\,]$ ($p\ge1$) such that $f_n(t)\le g(t)$ for all $n\in\mathbb{N}$ and almost all $t\in[a,b\,]$. Then the sequence of functions $f_n(t)$ has equi-absolutely continuous integrals on $[a,b\,]$.
\end{lemma}
\begin{proof}
Let $e\subset[a,b\,]$ be a measurable set and $m(e)<\delta$. By G\"older's inequality
$$
\int\limits_e f_n(t)\,dt\le\left(\int\limits_e f_n^p(t)\,dt\right)^{1/p}\left(\int\limits_e dt\right)^{1/q}
\le\|g\|_p\,\delta^{1/q}\,,
$$
whence equi-absolute continuity readily follows.
\end{proof}

Of course, if under conditions of this Lemma we have pointwise convergence or convergence almost everywhere $f_n(t)$ to a function $f(t)$ then Lebesgue's dominated convergence theorem gives at once $f(t)\in L_p$. But such convergence doesn't take place generally even under conditions of the Theorem~(\ref{criterion_of_absolutely_continuity}) as it follows for instance from the following example.

Let $f_n(t)=1+\cos(nt)$. We have uniformly in $x\in[0;1]$
$$
\int\limits_0^x(1+\cos(nt))\,dt\to\int\limits_0^x dt
$$
as $n\to\infty$. But the limit $\lim\limits_{n\to\infty} f_n(t)$ exists nowhere on $[0;1]$ except $t=0$.

Nevertheless one can state the following almost obvious assertion

\begin{theorem}
\label{condition_L_p}
Assume that there exists a summable on $[a,b\,]$ non-negative function $f(t)$ such that $\forall x\in[a,b\,]$
$$
\lim\limits_{n\to\infty}\int\limits_a^x f_n(t)\,dt=\int\limits_a^x f(t)\,dt
$$
Assume also that the sequence $f_n(t)$ is dominated by a positive function $g(t)\in L_p[a,b\,]$ ($p\ge1$), that is $f_n(t)\le g(t)$ for all $n\in\mathbb{N}$ and almost all $t\in[a,b\,]$. Then $f(t)\in L_p[a,b\,]$.
\end{theorem}
\begin{proof}
By reverse Fatou's lemma we have for any $[\alpha,\beta]\subseteq[a,b\,]$
$$
\int\limits_\alpha^\beta f(t)\,dt=\lim\limits_{n\to\infty}\int\limits_\alpha^\beta f_n(t)\,dt\le
\int\limits_\alpha^\beta \limsup\limits_{n\to\infty}f_n(t)\,dt\le\int\limits_\alpha^\beta g(t)\,dt
$$
From this it follows for $h>0$ and $x\in[a,b\,]$ ($f(t)=g(t)\equiv0$ for $t>b$)
$$
\frac{1}{h}\int\limits_x^{x+h}f(t)\,dt\le\frac{1}{h}\int\limits_x^{x+h}g(t)\,dt
$$
Passing to the limit $h\to 0$ in this inequality, we find almost everywhere on $[a,b\,]$
$$
f(x)\le g(x),
$$
and hence $f(t)\in L_p[a,b\,]$.
\end{proof}

Till now we supposed that $[a,b\,]\subseteq\sigma(A)$ but it is easy to give a sufficient condition for that in terms of functions $f_n(x)$. Actually by the same way as it was made in the Theorem~(\ref{condition_L_p}) one can prove that if $f_n(t)\ge C>0$ for all $n\in\mathbb{N}$ and almost all $t\in[a,b\,]$, then $f(t)\ge C>0$ for almost all $t\in[a,b\,]$ and hence $\sigma(A)$ is not empty on $[a,b\,]$.

Taking this into account and combining the results of Lemmas~(\ref{lemma2}), (\ref{equi-absolute_condition}) and Theorems~(\ref{criterion_of_absolutely_continuity}), (\ref{condition_L_p}), we obtain the following result which we will formulate for the brevity for the operators $A$ with simple spectrum

\begin{theorem}
\label{main_theorem}
Let $\{A_n\}_1^\infty$ be a sequence of linear self-adjoint operators in a separable Hilbert space $H$ with absolutely continuous spectrum on some interval $[a,b\,]$ and $\Phi$ be some dense set in $H$. Suppose that on $\Phi$ there exists in a strong sense the limit $\lim\limits_{n\to\infty}A_n$ and the closure of this limit operator is self-adjoint operator A with simple spectrum and generating vector $e_0$. Let $E_\lambda^{(n)}$ and $E_\lambda$ be the resolutions of the identity of the operators $A_n$ and $A$ respectively and let
$$
\sigma_n(\lambda)=(E_\lambda^{(n)} e_0, e_0)=\sigma_n(a)+\int\limits_{a}^\lambda f_n(t)\,dt\,,\quad \lambda\in[a,b\,]
$$
$$
\sigma(\lambda)=(E_\lambda e_0, e_0)\,,\quad \lambda\in[a,b\,]
$$
If there exist a constant $C>0$ and a positive function $g(t)\in L_p[a,b\,]$ ($p\ge1$) such that $C\le f_n(t)\le g(t)$ for all $n\in\mathbb{N}$ and almost all $t\in[a,b\,]$, then the spectrum of the operator $A$ on $[a,b\,]$ is purely absolutely continuous,
$$
\sigma(\lambda)=\sigma(a)+\int\limits_{a}^\lambda f(t)\,dt\,,\quad \lambda\in[a,b\,]
$$
and $f(t)\in L_p[a,b\,]$. Under these conditions the sequence $\sigma_n(\lambda)$ uniformly converges to $\sigma(\lambda)$ on $[a,b\,]$.
\end{theorem}

{\bf Remark.} {\it If the functions $f_n(t)$ are continuous and converge uniformly on $[a,b\,]$ to a positive function $f(t)$, then the spectrum of $A$ on $[a,b\,]$ is of course also purely absolutely continuous and $f(t)\in C[a,b]$.}

\section{Application to Jacobi matrices.}

Jacobi matrix is the tridiagonal matrix of the form
\begin{equation}
\label{Jacobi_Matrix}
  \left(\begin{array}{ccccc} a_0&b_0&0&0&\ldots\\
                      b_0&a_1&b_1&0&\ldots\\
                      0&b_1&a_2&b_2&\ldots\\
                      0&0&b_2&a_3&\ldots\\
                      \vdots&\vdots&\vdots&\vdots&\ddots
  \end{array}\right),
\end{equation}
where all the $a_n$ are real and all the $b_n$ positive.

An operator in separable Hilbert space $H$ can be assosiated with this matrix as follows~\cite{6}. Let $\{e_n\}_0^{\infty}$ be an orthonormal basis in $H$. Let us define basically the operator $A$ on the basis vectors $e_n$, according to the matrix representation~(\ref{Jacobi_Matrix}):
\begin{equation}
\label{Operator's_action}
\begin{array}{l}
  Ae_n=b_{n-1}e_{n-1}+a_ne_n+b_ne_{n+1}\,,\quad n=1,2,3,\ldots\\
  Ae_0=a_0e_0+b_0e_1
\end{array}
\end{equation}

Further, the operator $A$ is defined by linearity on all finite vectors, the set of which is dense in $H$.
(Finite vector has a finite number of non-zero components in the basis $\{e_n\}_0^{\infty}$.) Due to symmetricity of Jacobi matrix, for any two finite vectors $f$ and $g$ we have
$$
(Af,g)=(f,Ag)
$$

Therefore, the operator $A$ is symmetric and permits a closure. This minimal closed operator is a subject of the present consideration and it is naturally to save for it the notation $A$.

The operator $A$ can have deficiency indices $(1,1)$ or $(0,0)$. In the second case having the main interest in the present paper, operator $A$ is self-adjoint. A simple sufficient condition for self-adjointness of $A$ is Carleman's condition
\begin{equation}
\label{Charleman's_condtion}
\sum\limits_{n=0}^\infty\frac{1}{b_n}=+\infty
\end{equation}
In the self-adjoint case the spectrum of $A$ is simple and $e_0$ is generating element. The information about the spectrum of a self-adjoint operator $A$ is contaned in function
\begin{equation}
\label{resolvent}
R(\lambda)=((A-\lambda E)^{-1}e_0,e_0)=\int\limits_{-\infty}^{\infty}\frac{d\sigma(t)}{t-\lambda}\,,\quad \sigma(t)=(E_{t}e_0,e_0)\,,
\end{equation}
defined at $\lambda\not\in\sigma(A)$. ($E_{t}$ is the resolution of the identity of the operator $A$). The distribution function $\sigma(t)$ contains complete information about the spectrum. It has the following representation
$$
\sigma(\lambda)=\sigma_d(\lambda)+\sigma_{ac}(\lambda)+\sigma_{sc}(\lambda)\,,
$$
where $\sigma_d(\lambda)$ is a saltus function defining eigenvalues, $\sigma_{ac}(\lambda)$ and $\sigma_{sc}(\lambda)$ are the absolutely continuous and the singularly continuous parts respectively.

For the function $R(\lambda)$ there exists an afficient approximation algorithm based on the continued fractions. Namely, the following representation is valid
\begin{equation}
\label{c.fr.representation}
R(\lambda)=\frac{1}{\displaystyle a_0-\lambda-\frac{(b_0)^2}
{\displaystyle a_1-\lambda-\displaystyle
  \frac{(b_1)^2}{\displaystyle a_2-\lambda-\frac{(b_2)^2}{\displaystyle
  a_3-\lambda-\ldots}}}}
\end{equation}
The possibility of such representation and also the conditions and character of convergence are defined by the following theorem~\cite{6}

\begin{theorem}{(Hellinger)}
\label{Hellinger}
{ If the operator $A$ is self-adjoint then the continued fraction in~(\ref{c.fr.representation}) uniformly converges in any closed bounded domain of $\lambda$ without common points with real axis to the analytic function defined by formula~(\ref{resolvent}).}
\end{theorem}

The continued fraction convergence here means the existence of the finite limit of the $n$th approximant
$$
  \frac{1}{\displaystyle a_0-\lambda-\frac{b_0^2}{\displaystyle a_1-
            \lambda-\displaystyle
  \frac{b_1^2}{\displaystyle a_2-\lambda-\ldots-\frac{b_{n-2}^2}{
  \displaystyle
  a_{n-1}-\lambda}}}}=-\frac{Q_n(\lambda)}{P_n(\lambda)}\,,
$$
where $P_n(\lambda)$ and $Q_n(\lambda)$ are 1st and 2nd type polynomials respectively. These polynomials form a pair of linearly independent solutions of a second order finite difference equation
\begin{equation}
\label{recurrent_relations}
b_{n-1}\,y_{n-1}+a_n\,y_n+b_n\,y_{n+1}=\lambda\,y_n\,,\quad (n=1,2,3,\ldots)\,,
\end{equation}
with initial conditions
\begin{equation}
\label{initial_conditions}
P_0(\lambda)=1\,,\quad P_1(\lambda)=\frac{\lambda-a_0}{b_0}\,,\quad
Q_0(\lambda)=0\,,\quad Q_1(\lambda)=\frac{1}{b_0}
\end{equation}
The following equality is also valid
\begin{equation}
\label{vronsqian}
P_{n-1}(\lambda)\,Q_n(\lambda)-P_n(\lambda)\,Q_{n-1}(\lambda)=\frac{1}{b_{n-1}}\,,\quad (n=1,2,3,\ldots)
\end{equation}
Thus
$$
R(\lambda)=-\lim\limits_{n\to\infty}\frac{Q_n(\lambda)}{P_n(\lambda)}\,,\quad \mbox{Im}\,\lambda\ne0
$$

Exactly infinite part of the continued fraction (the behavior of $a_n$ and $b_n$ at infinity) defines essentially the spectrum of $A$. Particularly, an absolutely continuous part of spectrum is determined by the behavior of $a_n$ and $b_n$ at infinity too. Therefore we need to allocate an infinite part of the continued fraction starting from arbitrary place.

\begin{lemma}
\label{lemma1}
Denote by $K_n(\lambda)$ the infinite part ("tail") of continued fraction~(\ref{c.fr.representation}) starting with $n$th element of sequences $a_n$ and $b_n$
\begin{equation}
\label{fracion_for_K_n}
K_n(\lambda)=\frac{1}{\displaystyle a_n-\lambda-\frac{(b_n)^2}
{\displaystyle a_{n+1}-\lambda-\displaystyle
  \frac{(b_{n+1})^2}{\displaystyle a_{n+2}-\lambda-\frac{(b_{n+2})^2}{\displaystyle
  a_{n+3}-\lambda-\ldots}}}}
\end{equation}
Then
$$
R(\lambda)=-\frac{Q_n(\lambda)+Q_{n-1}(\lambda)\,b_{n-1}K_n(\lambda)}{P_n(\lambda)+P_{n-1}(\lambda)\,b_{n-1}K_n(\lambda)}\,,\quad
(n=1,2,3,\ldots)
$$
\end{lemma}

\begin{proof} The proof is obtained by induction. The validity of this formula for $n=1$ follows from initial conditions ~(\ref{initial_conditions}) for polynomials  $P_n$ and $Q_n$. Assuming the validity of formula
$$
R(\lambda)=-\frac{Q_{n-1}+Q_{n-2}\,b_{n-2}K_{n-1}}{P_{n-1}+P_{n-2}\,b_{n-2}K_{n-1}}\,,
$$
let us prove the same one for the value of index greater on unit. Expressing $K_{n-1}$ through $K_n$, we obtain
$$
R(\lambda)=-\frac{\displaystyle Q_{n-1}+Q_{n-2}\,b_{n-2}\,\frac{1}{\displaystyle a_{n-1}-\lambda-(b_{n-1})^2\,K_{n}}}
{\displaystyle P_{n-1}+P_{n-2}\,b_{n-2}\,\frac{1}{\displaystyle a_{n-1}-\lambda-(b_{n-1})^2\,K_{n}}}=
$$
$$
=-\frac{(a_{n-1}-\lambda)Q_{n-1}+Q_{n-2}\,b_{n-2}-Q_{n-1}\,(b_{n-1})^2\,K_{n-1}}
{(a_{n-1}-\lambda)P_{n-1}+P_{n-2}\,b_{n-2}-P_{n-1}\,(b_{n-1})^2\,K_{n-1}}
$$
Using~(\ref{recurrent_relations}), we find
$$
R(\lambda)=-\frac{-Q_n\,b_{n-1}-Q_{n-1}\,(b_{n-1})^2\,K_{n-1}}
{-P_n\,b_{n-1}-P_{n-1}\,(b_{n-1})^2\,K_{n-1}}=
-\frac{Q_n+Q_{n-1}\,b_{n-1}K_n}{P_n+P_{n-1}\,b_{n-1}K_n}
$$
Thus the formula is valid for all natural $n$.
\end{proof}

To construct an approximating sequence of operators $A_n$ one can choose the infinite part $K_n(\lambda)$ such that the spectrum of $A_n$ is absolutely continuous on some interval $[a,b\,]$ and, on the other hand, $K_n(\lambda)$ is rather simple and convenient for given $A$.

Choose the approximating sequence $A_n$ as follows
\begin{equation}
\label{A_n}
A_n=\left(\begin{array}{cccccccc} a_0&b_0&\ldots&0&0&0&0&\ldots\\
                      b_0&a_1&\ldots&0&0&0&0&\ldots\\
                      \vdots&\vdots&\ddots&\vdots&\vdots&\vdots&\vdots&\vdots\\
                      0&0&\ldots&a_{n-1}&b_{n-1}&0&0&\ldots\\
                      0&0&\ldots&b_{n-1}&a_n&b_n&0&\ldots\\
                      0&0&\ldots&0&b_n&a_n&b_n&\ldots\\
                      0&0&\ldots&0&0&b_n&a_n&\ldots\\
                      \vdots&\vdots&\vdots&\vdots&\vdots&\vdots&\vdots&\ddots
  \end{array}\right)\,,\quad (n=0,1,\ldots)
\end{equation}
The matrix $A_n$ differs from $A$ by that the elements of the diagonals are not changed starting with $n$rh number. If we denote the elements of the diagonal sequences of $A_n$ by $a^{(n)}_k$ and $b^{(n)}_k$, then
$$
a_k^{(n)}=\left[\begin{array}{c} a_k\,,\quad k<n\\
                                 a_n\,,\quad k\ge n
                \end{array}\right.\,,\quad\quad\quad
                b_k^{(n)}=\left[\begin{array}{c} b_k\,,\quad k<n\\
                                 b_n\,,\quad k\ge n
                \end{array}\right.
$$

Such matrices $A_n$ were used for approximation in works~\cite{4,17,18,19}. Note that they were used also in work~\cite{11}, where the absolutely continuous spectrum of Jacobi matrices also was considered, but in connection with commutation relations rather than with approximation.

It is clear that the matrices $A_n$ describe a sequence of bounded ( $D(A_n)=H$ ) self-adjoint operators (Carleman's condition~(\ref{Charleman's_condtion}) is obviously fulfilled) and it is natural to save the same notation for it.

The sequence $A_n$ strongly converges to the operator $A$ on the set of finite vectors since for any finite vector the difference $A-A_n$ vanishes, provided $n$ is sufficiently large. To use the Theorem~(\ref{main_theorem}) we should state the conditions under which the spectrum of $A_n$ is absolutely continuous. Note that the spectrum of the operators $A_n$ and $A$ is simple and all information about the spectrum contains in the functions $\sigma_n(\lambda)=(E_\lambda^{(n)} e_0, e_0)$ and $\sigma(\lambda)=(E_\lambda e_0, e_0)$ (generating vector $e_0$ is the same for $A$ and $A_n$). Below we will state the absolute continuity conditions of the spectrum of $A_n$. We will give also an explicit expression for the functions $f_n(t)=\sigma'_n(t)$. It is possible to do with the help of continued fraction approach.

Denote by $R_n(\lambda)$ the function $R(\lambda)$ for the operator $A_n$ defined by~(\ref{resolvent}).
Applying the Lemma~(\ref{lemma1}) and taking into account the definition of the operators $A_n$, we can present $R_n(\lambda)$ in the form
\begin{equation}
\label{R_n}
R_n(\lambda)=-\frac{Q_n(\lambda)+Q_{n-1}(\lambda)\,b_{n-1}K_n(\lambda)}{P_n(\lambda)+P_{n-1}(\lambda)\,b_{n-1}K_n(\lambda)}\,,
\end{equation}
where
\begin{equation}
\label{hvost}
K_n(\lambda)=\frac{1}{\displaystyle a_n-\lambda-\frac{(b_n)^2}
{\displaystyle a_{n}-\lambda-\displaystyle
  \frac{(b_{n})^2}{\displaystyle a_{n}-\lambda-\frac{(b_{n})^2}{\displaystyle
  a_{n}-\lambda-\ldots}}}}
\end{equation}

Continued fraction for $K_n(\lambda)$ describes an infinite part of the matrix of the operators $A_n$ in which the elements of the sequences are constant and equal $a_n$ and $b_n$. The function $K_n(\lambda)$ can be regarded as a matrix element of the resolvent of the operator represented by Jacobi matrix $J_n$ with constant elements along the diagonals
$$
J_n=\left(\begin{array}{ccccc} a_n&b_n&0&0&\ldots\\
                      b_n&a_n&b_n&0&\ldots\\
                      0&b_n&a_n&b_n&\ldots\\
                      0&0&b_n&a_n&\ldots\\
                      \vdots&\vdots&\vdots&\vdots&\ddots
  \end{array}\right)
$$
This Jacobi matrix is naturally connected with Tchebychev's polynomials, and the function $K_n(\lambda)$ can be found explicitly.
Note first that the operator $J_n$ is self-adjoint and by the theorem~(\ref{Hellinger}) the continued fraction~(\ref{hvost}) converges  to the function $K_n(\lambda)$ for any non-real $\lambda$. Write next the identity for the $k$th and $(k-1)$th approximants which follows readily from the structure of the continued fraction~(\ref{hvost})
$$
K_n^{(k)}(\lambda)=\frac{1}{\displaystyle a_n-\lambda-(b_n)^2\,K_n^{(k-1)}(\lambda)}
$$
Letting here $\mbox{Im}\,\lambda\ne0$ and passing to the limit, $k\to\infty$, one obtains the simple equation for $K_n(\lambda)$
$$
K_n(\lambda)=\frac{1}{\displaystyle a_n-\lambda-(b_n)^2\,K_n(\lambda)}\,,
$$
which has a solution
$$
K_n(\lambda)=\frac{a_n-\lambda+\sqrt{(a_n-\lambda)^2-4(b_n)^2}}{2\,(b_n)^2}
$$
One of the two branches of the square root is chosen from additional condition
$$
\mbox{Im}\,K_n(\lambda)>0\,,\quad\mbox{Im}\,\lambda>0\,,
$$
which is valid for a matrix element of the resolvent of any self-adjoint operator~\cite{1}. Using this one obtains for any $x\in\mathbb R$
\begin{equation}
\label{limit_K_n1}
\lim\limits_{\epsilon\to+0}K_n(x+i\epsilon)=D_n(x)+i\,B_n(x)
\end{equation}
\begin{equation}
\label{limit_K_n2}
D_n(x)=\left[\begin{array}{cr}\frac{\displaystyle a_n-x}{\displaystyle 2\,(b_n)^2}\,,& |a_n-x|\le2b_n\\
                             \frac{\displaystyle a_n-x\pm\sqrt{(a_n-x)^2-4(b_n)^2}}{\displaystyle 2\,(b_n)^2}\,,& |a_n-x|\ge2b_n
             \end{array}\right.
\end{equation}
\begin{equation}
\label{limit_K_n3}
B_n(x)=\left[\begin{array}{cr}\frac{\displaystyle\sqrt{4(b_n)^2-(a_n-x)^2}}{\displaystyle 2\,(b_n)^2}\,,& |a_n-x|\le2b_n\\
                             0\,,& |a_n-x|\ge2b_n
             \end{array}\right.
\end{equation}
In the expression for $D_n(x)$ the sign $"$+$"$ before the square root is taken when $x-a_n>2b_n$, the sing $"$$-$$"$ when $a_n-x>2b_n$.

From~(\ref{limit_K_n1}) - (\ref{limit_K_n3}) it follows that the spectrum of $J_n$ is absolutely continuous and concentrated on the interval $\sigma(J_n)=[a_n-2b_n,a_n+2b_n]$. The function $K_n(\lambda)$ is the matrix element of the resolvent of self-adjoint operator $J_n$. Hence for the function $K_n(\lambda)$ one has the integral representation
$$
K_n(\lambda)=\int\limits_{-\infty}^{\infty}\frac{f_{J_n}(x)\,dx}{x-\lambda}=
\frac{1}{2\pi (b_n)^2}\int\limits_{a_n-2b_n}^{a_n+2b_n}\frac{\sqrt{4(b_n)^2-(a_n-x)^2}\,dx}{x-\lambda}\,,
$$
where
\begin{equation}
\label{f_J_n}
f_{J_n}(x)=\left[\begin{array}{cr}\frac{\displaystyle\sqrt{4(b_n)^2-(a_n-x)^2}}{\displaystyle 2\pi\,(b_n)^2}\,,& |a_n-x|\le2b_n\\
                             0\,,& |a_n-x|>2b_n
             \end{array}\right.
\end{equation}
The function $f_{J_n}(x)$ at $a_n=0$, $b_n=1/2$ is the spectral weight for Tchebychev's polynomials~\cite{8,9}. The element $a_n$ defines the center of the spectrum and $b_n$ defines the width of the spectrum. The center and the width of the spectrum change according to the behavior of $a_n$ and $b_n$.

Prove that under some conditions the spectrum of $A_n$ is the same as the spectrum of $J_n$.

\begin{theorem}
\label{spectr_A_n}
 Inside the interval $I_n=[a_n-2b_n,a_n+2b_n]$ the spectrum of the operator $A_n$ is always absolutely continuous. The spectral weight $f_n(x)$ with accuracy to the values on a set of measure zero is defined by
\begin{equation}
\label{f_n(x)_3.6}
f_n(x)=\frac{\frac{\displaystyle 1}{\displaystyle \pi}\,B_n(x)}{P^2_n(x)-\frac{\displaystyle b_{n-1}}{\displaystyle b_n}P_{n-1}(x)P_{n+1}(x)}=
\frac{f_{J_n}(x)}{P^2_n(x)-\frac{\displaystyle b_{n-1}}{\displaystyle b_n}P_{n-1}(x)P_{n+1}(x)}\,,
\end{equation}
where $f_{J_n}(x)$ is spectral weight of $J_n$.

If for any $n\in\mathbb N$ the inequalities
$$
b_{n}\ge b_{n-1}\,,\quad |a_n-a_{n-1}|\le2(b_n-b_{n-1})
$$
hold, then the spectrum of $A_n$ is purely absolutely continuous, concentrated on the interval $I_n$ and $I_n\subseteq I_{n+1}$.
\end{theorem}

\begin{proof} Prove first that the absolutely continuous part of the spectrum of $A_n$ coincides with the spectrum of
$J_n$. For this purpose transform the expression~(\ref{R_n}) for $R_n(\lambda)$
$$
R_n(\lambda)=-\frac{Q_{n-1}(\lambda)}{P_{n-1}(\lambda)}-\frac{Q_{n}(\lambda)\,P_{n-1}(\lambda)-Q_{n-1}(\lambda)\,P_{n}(\lambda)}
{P_{n-1}(\lambda)\,(P_n(\lambda)+P_{n-1}(\lambda)\,b_{n-1}K_n(\lambda))}
$$
or, using~(\ref{vronsqian}),
$$
R_n(\lambda)=-\frac{Q_{n-1}(\lambda)}{P_{n-1}(\lambda)}-\frac{1}{b_{n-1}\,P_{n-1}(\lambda)\,
(P_n(\lambda)+P_{n-1}(\lambda)\,b_{n-1}K_n(\lambda))}
$$
Let $\lambda=x+i\epsilon\,,\:x\in\mathbb R$. Using~(\ref{limit_K_n1}), (\ref{limit_K_n3}) one obtains
\begin{equation}
\label{limit_R_n}
\lim\limits_{\epsilon\to+0}\mbox{Im}\,R_n(x+i\epsilon)=\frac{B_n(x)}
{(P_n(x)+b_{n-1}\,P_{n-1}(x)\,D_n(x))^2+(b_{n-1}\,P_{n-1}(x)\,B_n(x))^2}
\end{equation}
It is easy to note that the denominator of the fraction~(\ref{limit_R_n}) nowhere within the interval of the absolutely continuous spectrum of $J_n$ is zero. Indeed, if $|x-a_n|<2b_n$, then the denominator may be zero if and only if $x$ is a common zero of $P_n$ and $P_{n-1}$. But this is impossible since $P_n$ and $P_{n-1}$ have no common zeros~\cite{6}.

Thus from the limit expression~(\ref{limit_R_n}) it follows readily that the absolutely continuous parts of the spectrum of $A_n$ and $J_n$ coincide. Prove that under the conditions of the theorem the singular part of the spectrum of $A_n$ is absent. For this purpose, return to the formula~(\ref{R_n}). The numerator and the denominator in this formula are analytic functions on the set
$\mathbb C\setminus\sigma(J_n)$. Therefore the point $x\in\mathbb R$ belongs to the singular part of the spectrum of $A_n$ for $|x-a_n|\le2b_n$ if and only if
\begin{equation}
\label{3.3}
\begin{array}{l}
\lim\limits_{\epsilon\to+0} (P_n(x+i\epsilon)+P_{n-1}(x+i\epsilon)\,b_{n-1}\,K_n(x+i\epsilon))=\\
=P_n(x)+P_{n-1}(x)\,b_{n-1}D_n(x)+i\,P_{n-1}(x)\,b_{n-1}B_n(x)=0
\end{array}
\end{equation}
and for $|x-a_n|>2b_n$ if and only if $x$ is a zero of the denominator:
\begin{equation}
\label{3.4}
P_n(x)+P_{n-1}(x)\,b_{n-1}K_n(x)=P_n(x)+P_{n-1}(x)\,b_{n-1}D_n(x)=0
\end{equation}
But as it was shown before, the equality~(\ref{3.3}) is impossible for $|x-a_n|<2b_n$. It remains to consider the case $|x-a_n|\ge2b_n$. In this case the conditions~(\ref{3.3}) and~(\ref{3.4}) coincide since $B_n(x)=0$.

Thus, let $|x-a_n|\ge2b_n$. From the definition (\ref{limit_K_n2}) of $D_n(x)$ it follows that $\forall x\in\mathbb R$
$$
|D_n(x)|\le\frac{1}{b_n}
$$
Using this estimate and the assumed monotonicity of the sequence $b_n$ one obtains
\begin{equation}
\label{3.5}
\begin{array}{l}
|P_n(x)+P_{n-1}(x)\,b_{n-1}D_n(x)|\ge|P_n(x)|-|P_{n-1}(x)|b_{n-1}|D_n(x)|\ge\\
\ge|P_n(x)|-|P_{n-1}(x)|\,\frac{\displaystyle b_{n-1}}{\displaystyle b_n}\ge |P_n(x)|-|P_{n-1}(x)|
\end{array}
\end{equation}
On the other hand, from the recurrence relations~(\ref{recurrent_relations}) for the polynomials $P_n$ one has
$$
|P_{n}(x)|=\left|\frac{(x-a_{n-1})P_{n-1}(x)-b_{n-2}P_{n-2}(x)}{b_{n-1}}\right|\ge\frac{|x-a_{n-1}|}{b_{n-1}}|P_{n-1}(x)|-
\frac{b_{n-2}}{b_{n-1}}|P_{n-2}|
$$
If $|x-a_n|\ge2b_n$, then by virtue of the theorem condition $|a_n-a_{n-1}|\le2(b_n-b_{n-1})$, the inequality $|x-a_{n-1}|\ge2b_{n-1}$ also holds. Therefore from the last inequality one obtains
$$
|P_{n}(x)|\ge2|P_{n-1}(x)|-|P_{n-2}(x)|
$$
or
$$
|P_{n}(x)|-|P_{n-1}(x)|\ge|P_{n-1}(x)|-|P_{n-2}(x)|
$$
Continuing this process by induction one obtains
$$
|P_{n}(x)|-|P_{n-1}(x)|\ge|P_{1}(x)|-|P_{0}(x)|\ge2-1=1
$$
Finally, substituting this inequality into~(\ref{3.5}) one arrives at the following estimate of the denominator in the formula~(\ref{R_n}) in the case
$|x-a_n|\ge2b_n$
$$
|P_n(x)+P_{n-1}(x)\,b_{n-1}D_n(x)|\ge1\,,
$$
whence it follows the absence in this area the points of a singular spectrum of $A_n$.

It remains to prove the formula~(\ref{f_n(x)_3.6}). Using~(\ref{f_J_n}) and (\ref{limit_R_n}) one has
$$
f_n(x)=\frac{1}{\pi}\lim\limits_{\epsilon\to+0}\mbox{Im}\,R_n(x+i\epsilon)=\frac{\frac{\displaystyle 1}{\displaystyle \pi}\,B_n(x)}
{(P_n(x)+b_{n-1}\,P_{n-1}(x)\,D_n(x))^2+(b_{n-1}\,P_{n-1}(x)\,B_n(x))^2}=
$$
$$
=\frac{f_{J_n}(x)}{(P_n(x)+b_{n-1}\,P_{n-1}(x)\,D_n(x))^2+(b_{n-1}\,P_{n-1}(x)\,B_n(x))^2}
$$
Transform the denominator. When $|x-a_n|\le2b_n$, using~(\ref{limit_K_n2}), (\ref{limit_K_n3}), (\ref{recurrent_relations}), one has
$$
(P_n(x)+b_{n-1}\,P_{n-1}(x)\,D_n(x))^2+(b_{n-1}\,P_{n-1}(x)\,B_n(x))^2=
$$
$$
=\left(P_n+\frac{a_n-x}{2b_n^2}b_{n-1}P_{n-1}\right)^2
+P_{n-1}^2b_{n-1}^2\frac{4b_n^2-(a_n-x)^2}{4b_n^4}=
$$
$$
=P_n^2+\frac{a_n-x}{b_n^2}b_{n-1}P_{n-1}P_n+\frac{b_{n-1}^2}{b_n^2}P_{n-1}^2=P_n^2+P_{n-1}\frac{b_{n-1}}{b_n^2}((a_n-x)P_n+
b_{n-1}P_{n-1})=
$$
$$
=P_n^2-\frac{b_{n-1}}{b_n}P_{n-1}P_{n+1}
$$
The theorem is proved.
\end{proof}

Additional condition of the theorem means some subordination of $a_n$ to $b_n$. Indeed, for $k\in{\mathbb N}$, one has
$$
|a_k|-|a_{k-1}|\le|a_k-a_{k-1}|\le2(b_k-b_{k-1})
$$
Adding these inequalities for $k=1,2,\ldots,n$, one obtains
$$
|a_n|\le|a_0|+2(b_n-b_0)
$$

Note also that the expression
$$
P_n^2(x)-\frac{b_{n-1}}{b_n}P_{n-1}(x)P_{n+1}(x)
$$
appears probably first time in work~\cite{20}. In~\cite{21} it is also called "Turan determinant" although the classical Turan determinant~\cite{24} is
$$
P_n^2(x)-P_{n-1}(x)P_{n+1}(x)=\left|\begin{array}{cc} P_n(x) & P_{n+1}(x)\\
                                                      P_{n-1}(x) & P_n(x)
                                    \end{array}\right|
$$

\vspace{0.8cm}

{\bf Definition.} {\it Let us call the system of intervals $I_n$ {\it centered} on some interval $[a,b\,]$ if this interval is contained inside the interval $I_n$ for all sufficiently large values of~$n$.}\\

Now, when the absolute continuity conditions for the spectrum of $A_n$ are stated, we can apply the theorem~(\ref{main_theorem}). Combining the results of the theorems~(\ref{main_theorem}) and~(\ref{spectr_A_n}) and using the Definition, one obtains

\begin{theorem}
\label{Jacobi_criterion_of_absolutely_continuity}
Assume that the operator $A$ which is represented by~(\ref{Jacobi_Matrix}) is self-adjoint and the system of intervals $I_n$ is centered on some interval $[a,b\,]$. Let
$$
f_n(x)=\frac{\sqrt{4(b_n)^2-(a_n-x)^2}}{2\pi\,b_n\left[b_nP^2_n(x)-b_{n-1}P_{n-1}(x)P_{n+1}(x)\right]}\,,\quad x\in [a,b\,]
$$
Assume that there exist a constant $C>0$ and a positive function $g(x)\in L_p[a,b\,]$ ($p\ge1$) such that for all sufficiently large $n$ and almost all $x\in[a,b\,]$
$$
C\le f_n(x)\le g(x)
$$
Then the spectrum of the operator $A$ is purely absolutely continuous on $[a,b\,]$. The sequence $\sigma_n(\lambda)$ uniformly on $[a,b\,]$ converges to $\sigma(\lambda)$,
$$
\sigma(\lambda)=\sigma(a)+\int\limits_a^\lambda f(x)\,dx\,,
$$
and $f(x)\in L_p[a,b\,]$.
\end{theorem}

{\bf Remark.} {\it The functions $f_n(x)$ are continuous on $[a,b\,]$, if the system $I_n$ is centered on this interval. Hence if the system $f_n(x)$ converges uniformly on $[a,b\,]$ to a positive function $f(x)$, then the spectrum of $A$ on $[a,b\,]$ is purely absolutely continuous and $f(x)\in C[a,b]$.}

Using the recurrence relations for polynomials $P_n$, one can present the denominator of functions $f_n(x)$ in the form
$$
b_nP^2_n(x)-b_{n-1}P_{n-1}(x)P_{n+1}(x)=b_n(P_{n+1}^2(x)+P_{n}^2(x))-(x-a_n)P_{n+1}(x)P_{n}(x)
$$
Hence if $|x-a_n|\le 2b_n q$ where $0<q<1$, then 
$$
b_n\left(P_{n+1}^2(x)+P_{n}^2(x)\right)(1-q)\le b_nP^2_n(x)-b_{n-1}P_{n-1}(x)P_{n+1}(x)\le b_n\left(P_{n+1}^2(x)+P_{n}^2(x)\right)(1+q)
$$
Therefore
$$
\frac{\frac{\sqrt{1-q^2}}{\pi(1+q)}}{b_n(P_{n+1}^2(x)+P_{n}^2(x))}\le f_n(x)\le\frac{\frac{1}{\pi(1-q)}}{b_n(P_{n+1}^2(x)+P_{n}^2(x))}\,,
$$
whence one obtains

\begin{theorem}
\label{absolute_continuity_A_2}
Suppose that the operator $A$ represented by Jacobi matrix~(\ref{Jacobi_Matrix}) is self-adjoint.
Suppose that there exist the constants $C>0$, $0<q<1$ and a positive function $g(x)\in L_p[a,b\,]$ ($p\ge1$) such that for all sufficiently large $n$ and almost all $x\in[a,b\,]$ the following estimates hold\\

$1.\quad|x-a_n|\le 2b_n q$

\vspace{0.2cm}
$2.\quad\frac{\displaystyle 1}{\displaystyle g(x)}\le b_n(P_{n+1}^2(x)+P_{n}^2(x))\le C$
\vspace{0.2cm}

Then the spectrum of the operator $A$ is purely absolutely continuous on $[a,b\,]$. The sequence $\sigma_n(\lambda)$ uniformly on $[a,b\,]$ converges to $\sigma(\lambda)$,
$$
\sigma(\lambda)=\sigma(a)+\int\limits_a^\lambda f(x)\,dx\,,
$$
and $f(x)\in L_p[a,b\,]$.
\end{theorem}

Note that the right hand side of the second estimate of this theorem is a consequence of subordination theory as it was shown in~\cite{10,12} (see also~\cite{13,14,15}) and it provides the presence of absolutely continuous part of the spectrum on $[a,b\,]$. Here it is obtained by a natural and simple way which is not connected with subordination theory.

The most simple form the second estimate has in the case $L_\infty$. Actually then it takes the form
$$
C_1\le b_n(P_{n+1}^2(x)+P_{n}^2(x))\le C_2\,
$$
where $C_1,C_2$ are some positive constants. 

One can state a simple conditions on the coefficients $a_n$ and $b_n$ of Jacobi matrix in the case $f(x)\in C(-\infty,+\infty)$ (with accuracy up to equivalence). Some kinds of this result were obtained by different methods in works~\cite{22,5,14,18}. Here we give independent, new and simple proof of it.
 
\begin{theorem}
\label{absolute_continuity_A_trough_a_n_b_n}
Suppose that the operator $A$ represented by Jacobi matrix~(\ref{Jacobi_Matrix}) is self-adjoint.
Let $a_n$ and $b_n$ satisfy the conditions
\begin{enumerate}
\item $\displaystyle \lim b_n=+\infty$

\item $\displaystyle \lim\frac{a_n}{b_n}=s\,,\qquad 0\le |s|<2$

\item $\displaystyle \lim\frac{b_n}{b_{n+1}}=1$

\item $\displaystyle\left\{\frac{b_{n-1}}{b_n}-\frac{b_{n-2}}{b_{n-1}}\right\}
\in l_1\,,\quad
\left\{\frac{1}{b_n}-\frac{1}{b_{n-1}}\right\}\in l_1\,,\quad
\left\{\frac{a_n}{b_n}-\frac{a_{n-1}}{b_{n-1}}\right\}\in l_1$

\end{enumerate}
Then the spectrum of the operator $A$ is purely absolutely continuous $\sigma(A)=\sigma_{ac}(A)=\mathbb{R}$, the corresponding spectral density $f(x)\in C(-\infty,+\infty)$ and the sequence $\sigma_n(x)$ uniformly converges to the function $\sigma(x)$ on any finite interval.
\end{theorem}

\begin{proof}
Fix an arbitrary finite interval $[a,b\,]$ of a real axis. From conditions (1) and (2) it follows that for all $x\in[a,b\,]$ and all sufficiently large $n$ the estimation $|x-a_n|\le 2b_n q$ holds for $|s|/2<q<1$. Denote
$$
\Delta_n(x)=b_nP^2_n(x)-b_{n-1}P_{n-1}(x)P_{n+1}(x)
$$
Then
$$
\lim\limits_{n\to\infty}f_n(x)=\frac{1}{\pi}\lim\limits_{n\to\infty}\frac{1}{\Delta_n(x)}\,,
$$
and due to Remark to the Theorem~(\ref{Jacobi_criterion_of_absolutely_continuity}) we have to prove that the sequence $\Delta_n(x)$ uniformly on $[a,b\,]$ converges to a positive function (recall that $\Delta_n(x)$ is always positive for all sufficiently large $n$ when the system $I_n$ is centered on $[a,b\,]$). We have
$$
\Delta_{n+1}-\Delta_n=\left(1-\frac{b_n}{b_{n+1}}\right)b_{n+1}P_{n+1}^2-\left(1-\frac{b_n}{b_{n+1}}\right)b_nP_n^2+
$$
$$
+\left(x\left(\frac{1}{b_n}-\frac{1}{b_{n+1}}\right)+\left(\frac{a_{n+1}}{b_{n+1}}-\frac{a_n}{b_n}\right)\right)b_n P_n P_{n+1}
$$
Adding consecutively this equalities, one obtains for $n>m+1$
$$
\Delta_n-\Delta_m=\gamma_n b_nP_n^2-\gamma_m b_{m}P_{m}^2+
\sum_{k=m+1}^{n-1}\alpha_k b_k P_k^2+\sum_{k=m}^{n-1}\beta_k b_k 2P_k P_{k+1}\,,
$$
where
$$
\alpha_k=\frac{b_k}{b_{k+1}}-\frac{b_{k-1}}{b_{k}}\,,\quad \beta_k=\frac{1}{2}\left(x\left(\frac{1}{b_k}-\frac{1}{b_{k+1}}\right)+\left(\frac{a_{k+1}}{b_{k+1}}-\frac{a_k}{b_k}\right)\right)\,,
\quad \gamma_k=1-\frac{b_{k-1}}{b_k}
$$
Since for all $n$ sufficiently large and all $x\in[a,b\,]$
$$
b_n P_n^2(1-q)\le b_n\left(P_{n+1}^2+P_{n}^2\right)(1-q)\le \Delta_n\,
$$
we have for $m$ sufficiently large ($n>m+1$) and all $x\in[a,b\,]$
\begin{equation}
\label{estimation_for_Delta}
(1-q)\left|\Delta_n-\Delta_m\right|\le\left|\gamma_n\right|\Delta_n+\left|\gamma_m\right|\Delta_{m}+
\sum_{k=m}^{n-1}(|\alpha_k|+|\beta_k|)\Delta_k
\end{equation}
Suppose that the sequence $\Delta_n$ is uniformly bounded in $[a,b\,]$ (we will prove it below). Then from conditions (3) and (4) of the Theorem it follows that for any $\epsilon>0$ and any $x\in[a,b\,]$
$$
|\Delta_n-\Delta_m|<\epsilon\,,
$$
provided $m>N$ and any $n>m$. In the case $n=m+1$ this estimation is also fulfilled. It follows that the sequence $\Delta_n$ uniformly converges to nonnegative continuous on $[a,b\,]$ function $\Delta(x)$.

It remains to prove uniform boundedness of $\Delta_n(x)$ and positivity of $\Delta(x)$. Return to the inequality~(\ref{estimation_for_Delta}). One has
$$
(1-\delta_n)\Delta_n\le(1+\delta_m)\Delta_m+\sum_{k=m}^{n-1}\epsilon_k\Delta_k\,,
$$
where
$$
\delta_k=\frac{|\gamma_k|}{1-q}\,,\quad \epsilon_k=\frac{|\alpha_k|+|\beta_k|}{1-q}
$$
Denote
$$
B_n=\sum_{k=m}^{n-1}\epsilon_k\Delta_k
$$
We have
$$
B_n-B_{n-1}=\epsilon_{n-1}\Delta_{n-1}\le\epsilon_{n-1}\frac{B_{n-1}+(1+\delta_m)\Delta_m}{1-\delta_{n-1}}
$$
and hence
$$
B_n\le\left(1+\frac{\epsilon_{n-1}}{1-\delta_{n-1}}\right)B_{n-1}+\frac{\epsilon_{n-1}}{1-\delta_{n-1}}(1+\delta_m)\Delta_m
$$
From this we consecutively deduce 
$$
B_n\le(B_m+(1+\delta_m)\Delta_m)\prod_{k=m}^{n-1}\left(1+\frac{\epsilon_{k}}{1-\delta_{k}}\right)\le
$$
$$
\le(B_m+(1+\delta_m)\Delta_m)\prod_{k=m}^{\infty}\left(1+\frac{\epsilon_{k}}{1-\delta_{k}}\right)=C<+\infty
$$ 
This estimation is uniform in $x\in[a,b\,]$. Here we used again the uniform convergence of series $\sum\frac{\epsilon_{k}}{1-\delta_{k}}$ which follows from conditions (3) and (4) of the Theorem. From
$$
(1-\delta_n)\Delta_n\le(1+\delta_m)\Delta_m+B_n
$$
it follows that $\Delta_n$ is also uniformly bounded.

Let us prove at last that the limit function $\Delta(x)$ is strictly positive in $[a,b\,]$. From~(\ref{estimation_for_Delta}) one has
$$
(1-\delta_n)\Delta_n\ge(1-\delta_m)\Delta_m-\sum_{k=m}^{n-1}\epsilon_k\Delta_k
$$
(in the same notations as before). Since $\sum\limits_{k=m}^{n-1}\epsilon_k\Delta_k\le\epsilon\Delta_m$ where $\epsilon$ is arbitrary small for $m$ sufficiently large, we have
$$
(1-\delta_n)\Delta_n\ge(1-\delta_m-\epsilon)\Delta_m=C>0
$$ 
for $m$ sufficiently large and all $n>m+1$. Passing to the limit $n\to\infty$, we obtain $\Delta\ge C>0$ ($C$ depends on $x$ of course).

Thus the sequence $f_n(x)$ uniformly converges to a positive continuous function $\displaystyle f(x)=\frac{1}{\pi\Delta(x)}$ in $[a,b\,]$ and Theorem~(\ref{Jacobi_criterion_of_absolutely_continuity}) gives the pure absolute continuity of the operator $A$ spectrum on $[a,b\,]$. Since the interval $[a,b\,]$ is arbitrary the theorem is proved.
\end{proof}

In work~\cite{16} it was shown that in the case $\displaystyle\lim\frac{a_n}{b_n}=s\,,\, 0<|s|<2$ one can extend the class of weights. Namely, the basic conclusions of the Theorem~(\ref{absolute_continuity_A_trough_a_n_b_n}) (except $f(x)\in C(-\infty,+\infty)$) remains valid if the sequences
$$
\frac{a_{n-1}a_n}{b_n^2}\,,\quad\frac{a_{n-1}+a_n}{b_n^2}\,,\quad\frac{1}{b_n^2}
$$
have bounded variation. Using the Theorem~(\ref{Jacobi_criterion_of_absolutely_continuity}) and main ideas of~\cite{16} one can show that in this case also $f(x)\in C(-\infty,+\infty)$.
 
\section{Conclusion.}

Note that the method of analysis of the absolutely continuous spectrum developed here can be applied not to the finite-difference operators only but to any self-adjoint operators in separable Hilbert space provided that one can find a convenient approximative sequence of operators $A_n$ with absolutely continuous spectrum strongly converging to the operator $A$ on a dense set.

\vspace{0.5cm}

If this manuscript was useful for you, you may support author because this project is not supported by Russian State:
yandex purse number - 410013774009102

\end{document}